\documentclass[11pt]{article}

\usepackage[usenames,dvipsnames,condensed]{xcolor}

\usepackage{amsthm}
\usepackage{amsfonts}
\usepackage[UKenglish]{babel}
\usepackage[usenames]{xcolor}
\usepackage{graphicx}
\usepackage{soul}
\usepackage{stfloats}
\usepackage{morefloats}
\usepackage{cite}
\usepackage{lscape}
\usepackage{epstopdf}

\usepackage{amsmath,amstext,amsopn,amsfonts,eucal,amssymb}
\usepackage{graphicx,wrapfig,url}

\newcommand\Z{{\mathbb Z}}

\newcommand\R{{\mathbb R}}

\newtheorem{theorem}{Theorem}[section]
\newtheorem{corollary}[theorem]{Corollary}
\newtheorem{lemma}[theorem]{Lemma}
\newtheorem{proposition}[theorem]{Proposition}
\newtheorem{definition}[theorem]{Definition}

\newtheorem{example}[theorem]{Example}

\newtheorem{remark}[theorem]{Remark}

\newcommand{\Exp}[2]{\mathbf{e}^{#1 \mathbf{#2}}}

\newcommand{\SO}{{\rm SO(3)}}
\newcommand{\SU}{{\rm SU(2)}}

\newcommand{\GA}{{\rm GAlex}}
\newcommand{\Eis}{{\rm Eis}}

\newcommand{\rot}[2]{\/{\rm Rot}_{#1}{(\mathbf{#2})}}

\newcommand{\mbf}[1]{\/ \mathbf{#1}}

\begin{document}

\title{Longitudinal Mapping Knot Invariant for $\SU$ }

\author{W.~Edwin Clark, \  Masahico Saito \\
Department of Mathematics and Statistics\\ University of South Florida
}

\date{\empty}

\maketitle

\begin{abstract}
The knot coloring polynomial defined by  Eisermann for a finite pointed group 
is generalized to an  infinite pointed group as the {\it longitudinal mapping invariant} of a knot. In turn this can be 
thought of as a generalization of the quandle 2-cocycle invariant for finite quandles.
If the group is a topological group then this  invariant can be thought of a topological 
generalization of the 2-cocycle invariant.
The longitudinal mapping  invariant is based on a meridian-longitude pair in the knot group. We also give an interpretation of the  invariant in terms of quandle colorings of a 1-tangle for generalized Alexander quandles without use of a meridian-longitude pair in the knot group.
The invariant values are concretely  evaluated for the torus knots $T(2,n)$, their mirror images,  and the figure eight knot for the group $\SU$.
\end{abstract}

\section{Introduction}
In this paper all knots will be oriented and we write equality for orientation preserving ambient isotopy.
For a knot $K$ we write $r(K)$ for $K$  with orientation reversed and $m(K)$ for the mirror image of $K$.
It is known that the knot quandle $\mathcal{Q}(K)$ distinguishes  distinct oriented knots $K_1$ and $K_2$ if and only if $K_2 \neq rm(K_1)$ (\cite{Joyce,Mat}\ ). The knot group $\pi_K$ cannot distinguish $K$ from $s(K)$ for any 
$s \in \{ r,m,rm \}$ (see, e.g.,\cite{BZ}). It follows that neither the set of (quandle) homomorphisms ${\rm Hom}_{\rm Qnd}(\mathcal{Q}(K), Q)$ from $\mathcal{Q}(K)$ to a quandle $Q$ nor the set ${\rm Hom}_{\rm Gp}(\pi_K,G)$ of (group) homomorphisms from $\pi_K$ to a group $G$ is a complete invariant of oriented knots. 

In the case of quandles a stronger invariant (the {\it 2-cocyle invariant} or {\it 2-cocycle state-sum invariant}) was obtained  in \cite{CJKLS} 
using a 2-cocycle $\varphi$ for a finite quandle $Q$ with coefficients in an abelian group $\Lambda$.  One defines a mapping
$$B_\varphi: {\rm Hom}_{\rm Qnd}(\mathcal{Q}(K), Q) \rightarrow \Lambda:\quad \rho \mapsto B_\varphi(\rho)$$ whose fibers determine a partition of  ${\rm Hom}_{\rm Qnd}(\mathcal{Q}(K), Q)$ indexed by $\Lambda$. Since $Q$ is finite this partition can be expressed as 
an element $\Phi_Q^\varphi(K)$ of the group ring $\Z[\Lambda]$. See \cite{CDS} for evidence that the 2-cocyle invariant might be a complete invariant for oriented knots.

In the case of groups the knot group {\it peripheral system} $(\pi_K, m_K, l_K)$,  where $(m_K, l_k)$ is a meridian-longitude pair, is a complete invariant of oriented knots (see \cite{BZ}). Using this, Eisermann \cite{Eis-colpoly} defined the {\it knot coloring polynomial} for a pointed finite group $(G,x)$ corresponding to a peripheral system $(\pi_K, m_K, l_K)$ as
   $$P_G^x(K) = \sum_{\rho} \rho (l_K) $$
 where the sum is taken over all homomorphism $\rho: \pi_K \rightarrow G$ with $\rho(m_K) = x$. 
It turns out that longitude images lie in $\Lambda = C(x) \cap G'$ and hence $P_G^x(K)$ is an element of the group ring $\Z[\Lambda]$.  Eisermann shows in \cite{Eis-colpoly} that  when $G$ is finite and $\Lambda$ is abelian a knot coloring polynomial can be expresssed as a 2-cocycle invariant over the conjugation quandle $x^G$ and conversely a 2-cocycle invariant for a finite quandle $Q$ is a specializations of a knot coloring
polynomial for $G = {\rm Inn}(\tilde{Q})$ where $\tilde{Q}$ is the abelian extension corresponding to the given 2-cocycle.  In particular, any knots distinguishable by 2-cocycle invariants are distinguishable by knot coloring polynomials. We note however that  in general the price one pays for this is a group much larger than the quandle.

In case $G$ is infinite the coefficients of $P_G^x(K)$ may be infinite, then  we replace it by the {\it longitudinal mapping}
$$ \mathcal{L}_G^x(K) : {\rm Hom}_{\rm Gp}(\pi_K, m_K; G,x) \rightarrow \Lambda, \quad \rho \mapsto \rho(l_K),$$
where ${\rm Hom}_{\rm Gp}(\pi_K, m_K; G,x)$ is the set of homomorphisms  $\rho: \pi_K \rightarrow G$ with $\rho(m_K) = x$.   If $G$ is a topological group, $\mathcal{L}_G^x (K)$ may be thought of as a topological analogue of the 2-cocycle invariant or the knot coloring polynomial. This is the invariant we examine for the case $G = SU(2)$ in this paper.  We find $\mathcal{L}_G^x (K)$ when $K$ is a torus knot $T(2,n)$ for odd $n \geq 3$ and when $K$ is the figure eight knot $4_1$.

Let $Q$ be any quandle (possibly infinite) and let 
$T$ be a 1-tangle diagram whose closure is the knot $K$. 
 Denote the initial arc of $T$ by $0$ and the terminal arc by $n$.
 For arbitrary fixed $e \in Q$ let ${\rm Col}_Q^e(T)$ denote the set of colorings  of $T$ by quandle $Q$ 
 such that $C(0)=e$. 
Furthermore, by Lemma 2.2 in \cite{CDS}, for $C \in {\rm Col}_Q^e(T)$,  $b = C(n)$  satisfies $R_b = R_e$. That is, $b$ lies the the fiber $F_e = {\rm inn}^{-1}(R_e)$.  We define the mapping
$$\Psi_Q^e(K) : {\rm Col}_Q^e (T)  \rightarrow F_e, \quad C \mapsto C(n).$$
In the appendix we  show that if $Q$  is the generalized Alexander Quandle ${\rm GAlex}(G',f_x)$ constructed from the pointed group $(G,x)$ where $f_x (u) = x^{-1}ux$, then $ \mathcal{L}_G^x(K)$ is equivalent to $\Psi_Q^e(K)$ where $e = 1$.  This gives a way to construct the longitudinal mapping without use of a meridian-longitude pair.

\section{Basic Definitions}

In this section we briefly review some definitions and examples. 
More details can be found, for example, in \cite{CKS}. 

If $X$ is a set with a binary operation $*$ the {\it right translation}  ${R}_a:X \rightarrow  X$, by $a \in X$, is defined
by ${ R}_a(x) = x*a$ for $x \in X$. 
The magma $(X,*)$ is a {\it quandle} if each right translation $R_a$ is an automorphism of $(X,*)$ and every element of $X$ is idempotent.
A {\it quandle homomorphism} between two quandles $X, Y$ is
 a map $f: X \rightarrow Y$ such that $f(x*_X y)=f(x) *_Y f(y) $, where
 $*_X$ and $*_Y$ 
 denote 
 the quandle operations of $X$ and $Y$, respectively.
 A {\it quandle isomorphism} is a bijective quandle homomorphism, and 
 two quandles are {\it isomorphic} if there is a quandle isomorphism 
 between them.
 The set of quandle homomorphisms from $X$ to $Y$ is denoted by ${\rm Hom}_{\rm Qnd}(X, Y)$.
 A quandle epimorphism $f: X \rightarrow Y$ is a {\it covering}~\cite{Eis-unknot}
 if $f(x)=f(y)$ implies $a*x=a*y$ for all $a, x, y \in X$. 
   
 For a quandle $(X,*)$, since $R_a$ for each $a \in X$ is an automorphism, one may define the binary
 operation $\bar{*}$ by $x\ \bar{*}\ y=R^{-1}_y(x)$. This gives a quandle structure on $X$, called the {\rm dual} quandle.
The subgroup of ${\rm Sym}(X)$ generated by the permutations ${ R}_a$, $a \in X$, is 
called the {\it {\rm inner} automorphism group} of $X$,  and is 
denoted by ${\rm Inn}(X)$. 
The map ${\rm inn}: X \rightarrow {\rm inn}(X) \subset {\rm Inn}(X)$
(which is a quandle under conjugation)
defined by ${\rm inn}(x)=R_x$ is called the {\it inner representation}. 
An inner representation is a covering.

A quandle is {\it indecomposable} if ${\rm Inn}(X)$ acts transitively on $X$. We use {\it indecomposable}
here rather than {\it connected} to avoid confusion with the topological sense of the word.
A quandle is {\it faithful} if the mapping ${\rm inn}: X \rightarrow  {\rm Inn}(X)$ is an injection.

As in Joyce \cite{Joyce}, given a group $G$ and 
and $f \in {\rm Aut}(G)$, a quandle operation is defined on $G$ by 
$x*y=f(xy^{-1}) y ,$ $ x,y \in G$.  We call such a quandle a {\it generalized Alexander quandle} and denote it by  $\GA(G,f)$.
If $G$ is abelian, such a quandle is  known as an {\it Alexander quandle} or {\it affine quandle}. 

\begin{figure}[htb]
\begin{center}
\includegraphics[width=3in]{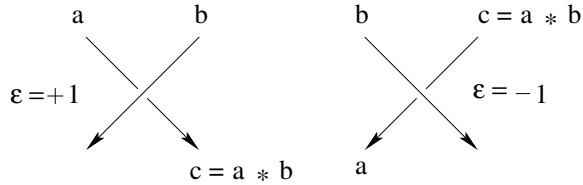}
\end{center}
\caption{Colored crossings, positive (left) and negative (right)}
\label{Xing}
\end{figure}

\bigskip

Let $D$ be a diagram of a knot $K$, and ${\cal A}(D)$ be the set of arcs of $D$.
A {\it  coloring}  of a knot diagram $D$ by a quandle $X$
is a map $C: {\cal A}(D) \rightarrow X$  satisfying the condition depicted in Figure~\ref{Xing}
at every 
 positive (left) and negative (right) crossing. 
respectively.  
The set of colorings of $D$ by $X$  is denoted by ${\rm Col}_X(D)$. There is a bijection from ${\rm Hom}_{\rm Qnd}(\mathcal{Q}(K), X) $ to ${\rm Col}_X(D)$. The cardinality $| {\rm Col}_X(D) |$ is a knot invariant
(e.g. see \cite{CKS}).

A   $1$-{\it tangle}, or a {\it long knot},  is a properly embedded arc in a $3$-ball, and the equivalence of
long knots 
is defined by ambient  isotopies of the $3$-ball fixing the
boundary.   
A diagram of a $1$-tangle  is defined in a
manner similar to a knot diagram, from a regular projection to a disk by
specifying  crossing information.  An orientation
of a $1$-tangle is specified by an arrow on a diagram.  A knot
diagram is obtained from a $1$-tangle diagram by closing the end points by a
trivial arc outside of a disk. This procedure is called the {\it closure} of a
$1$-tangle.  If a $1$-tangle is oriented, then the closure inherits the
orientation.  
Two  diagrams of the same $1$-tangle are related by Reidemeister moves.
There is a bijection between knots and $1$-tangles for classical knots, and invariants of 1-tangles give rise to
invariants of knots, see, for example, \cite{Eis-unknot,Nieb}. 

A quandle coloring of an oriented $1$-tangle diagram is defined in a  manner
similar to  those for knots.  We do not require that the end points receive the
same color for a quandle coloring of $1$-tangle diagrams. However this will be the case for a conjugation quandle.
For a quandle $Q$ and $x \in Q$, 
denote by  ${\rm Col}_Q^{x}(T)$ the set of colorings of a $1$-tangle $T$ by $Q$ with the initial arc colored by $x$.

\bigskip

\section{Computation of the Longitudinal Mapping} \label{sec:poly}

 For convenience we often identify the diagram of a tangle $T$ with the tangle itself.

\begin{definition}{\rm
({\it Wirtinger code} \,  Eisermann  \cite{Eis-unknot}) \label{Wirtinger}
Label the arcs  of a $1$-tangle $T$ by integers,  ${\cal A}(T)=\{ 0, \ldots, n \}$,  
such that $0$ and $n$ are  the initial and terminal  arcs, respectively,
and the remaining arcs are labeled in order when traveled along the tangle from $0$ to $n$.
At the end of arc number $i-1$, we undercross arc $\kappa(i) =\kappa i$ and continue on arc number $i$.
Let $\epsilon(i) = \epsilon i$ be the sign of crossing $i$.
Note that these are maps $\kappa : \{1,\dots, n \} \rightarrow \{0, \dots, n \}$ and $\epsilon : \{1, \dots, n \} \rightarrow \{1,-1\}.$  The pair $(\kappa,\epsilon)$ is called the {\it Wirtinger code} of the diagram $T$.
}
\end{definition}

The $1$-tangle group $\pi_T$ with diagram $T$ and Wirtinger code $(\kappa,\epsilon)$  allows the presentation 
$$ \pi_T = \langle x_0,x_1,\dots,x_n \ | \ r_1, \dots, r_n \rangle,  \text{ where } 
r_i  \text{ is the relation }  x_i = x_{\kappa i}^{-\epsilon i} x_{i-1}  x_{\kappa i}^{\epsilon i}.$$
As in \cite{Eis-colpoly} we  choose the meridian 
 \begin{eqnarray*}\label{meridian}
m_T = x_0
\end{eqnarray*}
 and the (preferred) longitude
\begin{eqnarray*}\label{long2}
l_T =
 x_0^{-w(T)} x_{\kappa 1} ^{\epsilon 1}x_{\kappa 2}^{\epsilon 2} \cdots x_{\kappa n} ^{\epsilon n}.
\end{eqnarray*}
See Remark 3.13 of \cite{BZ} for this form of the longitude. 
The  knot group $\pi_K$ is isomorphic to  $\pi_T$.

For a pointed {\it finite} group $(G,x)$, Eisermann defined the {\it knot coloring polynomial} of $K$ to be
$$
P_G^x(K) = \sum_{\rho} \rho(l_T),
$$
where the sum is taken over all homomorphisms $\rho: \pi_T \rightarrow G$ with $\rho(m_T) = x$.  It turns out (see \cite{Eis-colpoly}) that the values $\rho(l_T)$ lie in the {\it longitudinal group} 
$\Lambda = C(x) \cap G'$ where $C(x)$ is the centralizer of $x$ and $G'$ is the commutator subgroup of $G$. Thus $P^x_G(K)$ lies in the group ring $\Z[\Lambda]$. 

Let ${\rm Rep}_G^x(T)$ be the set of 
 homomorphisms  $\rho: \pi_T \rightarrow G$ with $\rho(m_T) = x$,
 and ${\rm Col}^x_Q(T)$ be the set of colorings ${ C}$ by a quandle $Q$ such that ${ C}(0)=x$, where 
 $0$ is the initial arc of $T$. 
There is a bijection between   $ {\rm Rep}_G^x(T)$ and  ${\rm Col}^x_Q(T)$ where $Q$ is the conjugacy class $x^G$ of $x$
 under the product $a*b = b^{-1}ab$.

We wish to extend Eisermann's knot coloring polynomial to groups not necessarily finite.

\begin{definition}
{\rm
Let $(G,x)$ be any pointed group. Let $K$ be a knot and $T$ be a $1$-tangle corresponding to  $K$.
We define the  knot invariant  $$\mathcal{L}_G^x(K): {\rm Rep}^x_G(T) \rightarrow \Lambda, \quad \rho \mapsto \rho(l_T). $$
We call it the {\em longitudinal mapping}.  When there is no chance of confusion we write $\mathcal{L}$ in place of $\mathcal{L}_G^x(K)$.  We shall say that two such longitudinal mappings
$\mathcal{L}_1: {\rm Rep}^x_G(T_1) \rightarrow \Lambda$ and
$\mathcal{L}_2: {\rm Rep}^x_G(T_2) \rightarrow \Lambda$ are {\it equivalent} if there is a bijection 
$\beta: {\rm Rep}^x_G(T_1)  \rightarrow  {\rm Rep}^x_G(T_2)$ such that $\mathcal{L}_1 = \mathcal{L}_2 \beta.$  Clearly the longitudinal mapping $\mathcal{L}:  {\rm Rep}_G^x(T) \rightarrow \Lambda$ 
is a knot invariant up to equivalence of mappings and if $G$ is a topological group $\mathcal{L}$ is continuous. In this case $\beta$ must be a homeomorphism.  See Rubinsztein \cite{Rub} for the topology on   ${\rm Col}_Q(T)$. 
}
\end{definition}

\begin{remark}{\rm
See the appendix for a definition of $\mathcal{L}$ that doesn't depend on the meridian-longitude pair $(m_T, l_T)$.}
\end{remark}

For a finite group $G$ the knot coloring polynomial is 
$P^x_G(K)=\sum_{v \in \Lambda} |\mathcal{L}^{-1}(v)| v$.
 Thus $\mathcal{L}$ can be seen as an analogue of the knot coloring polynomial
  defined for topological quandles. 
 Since the knot coloring polynomial is a generalization of the quandle $2$-cocycle invariant (see Theorem 3.24 in \cite{Eis-colpoly}), the invariant  $\mathcal{L}$ is a generalization of the quandle $2$-cocycle invariant. 
 See \cite{CDS,Eis-colpoly} for more details of relations among these invariants.
  A similar but different invariant  using longitudes was considered in \cite{Nieb}.

\begin{remark} 
{\rm 
Note that the group $\Lambda$ acts on the set of homomorphisms  $\rho: \pi_K \rightarrow G$ with $\rho(m_K) = x$ by setting $\rho^g(a) = g^{-1}\rho(a)g$ for $a \in G$. Since $g \in C(x)$ it follows that $\rho^g(m_K) = x$. Hence if $\Lambda $ is abelian, then  $\mathcal{L}$ is constant on the orbits of this action by $\Lambda$. 
In our application, $\Lambda $ is abelian.
 Thus for example for a two-bridge knot with diagram $T$, suppose the arcs $x_0$ and $x_1$ in the above notation are the two bridges. Then $\mathcal{L} (\rho)$ is completely determined by the values $\rho(x_0) = \rho(m_K) = x \in G$ and the values of $\rho(x_1)$.
}
\end{remark}

\begin{proposition} \label{mirror} $\mathcal{L}_G^x(rm(K))(\rho) = \mathcal{L}_G^x(K)(\rho)^{-1}$ for all $\rho \in  {\rm Rep}^x_G(T) $. \end {proposition}
\begin{proof} This is immediate from the fact that if $(\pi_K, m_K, l_K)$ is a peripheral system for knot $K$ then $(\pi_K, m_K, {l_K}^{-1})$ is a peripheral system for the knot $rm(K)$  (\cite{Kawauchi}, Chapter 6).
\end{proof}

\section{Background for  $\SU$ } \label{sec:SU(2)} 

For the remainder of the paper, we examine the invariant $\mathcal{L}$ for $(G,x) =(\SU,\mbf{x})$ with various choices of $\mbf{x}$.  We represent $\SU$ by the group of unit quaternions, that is,
$$ \SU =  \{ a + b\mathbf{i} + c\mathbf{j}+ d\mathbf{k}: a^2+b^2+c^2+d^2 = 1 \}.$$

The group $\SO$ will also be of use. Elements of $\SO$ will be denoted by  $\rot{\theta}{v}$, $\theta \in \R$, $\mathbf{v} \in S^2$. If $\mathbf{u} \in \R^3$, $\mathbf{u}\rot{\theta}{v}$ is the vector obtained by rotating $\mathbf{u}$ about $\mathbf{v}$ by $\theta$ radians  using the right-hand rule.

We represent elements of $\R^3$ as pure quaternions $\mathbf{u} = u_1\mathbf{i}+ u_2\mathbf{j}  +  u_3\mathbf{k}$ and we identify the set of pure unit quaternions with the sphere $S^2.$
Then each element of $\SU$ can be represented  the form 
$$ \Exp{\theta}{u} = \cos(\theta) + \sin(\theta)\mathbf{u}, \quad 
\mathbf{u} \in S^2,\quad 0 \leq \theta < 2\pi.$$ 

Note that a pure quaternion $\mathbf{u}$ satisfies $\mathbf{u}^2 = -1$ and hence the quaternions $\Exp{\theta}{u}$ for fixed $\mathbf{u}$ behave just like complex numbers $\Exp{\theta}{i}= \cos(\theta) + \sin(\theta) \mathbf{i}$. 

From \cite{DJ} (Section 1.2)  the conjugacy classes of $\SU$
are given by
$$
\tilde{C}_\theta =  \{ \Exp{\theta}{u} : \ \mathbf{u} \in S^2\},  
$$
for $0 \le  \theta \le \pi.$
In this case $\tilde{C}_0 = \{1\}$,  $\tilde{C}_{\pi} = \{-1\}$ and for $0 < \theta < \pi$, $\tilde{C}_\theta$ is a sphere.  This also follows from Lemma~\ref{lem:conjugation} below.

It is known (see for example \cite{Kuipers}, Theorem 5.1) that for $\mbf{u},\mbf{v} \in S^2$ and $\theta \in \R$ that
$$ \Exp{\theta}{u}\mathbf{v}\Exp{-\theta}{u} = \mathbf{v} \rot{2\theta}{u}.$$
The double covering homomorphism $\phi: \SU \rightarrow \SO$ may be defined by 
$$\mathbf{v} \phi(\mathbf{q})= \mathbf{q}^{-1}\mathbf{v}\mathbf{q}.$$
In this case if $\mathbf{q} = \Exp{\theta}{u}$, then  $\phi(\mathbf{q}) = \rot{-2\theta}{\mathbf{u}}$, the rotation by $-2\theta$ radians about the unit
vector $\mathbf{u}$. We must take $\phi(\mathbf{q})$ to be $\mathbf{q}^{-1}\mathbf{v}\mathbf{q}$ instead of $\mathbf{q}\mathbf{v}\mathbf{q}^{-1}$ since we write the rotation operator on the right of the argument. 

\begin{lemma}\label{lem:conjugation} For fixed $\theta, \beta \in  \R $  and $\mathbf{u,v} \in S^2$ we
have
   $${\Exp{-\beta}{v}}\Exp{\theta}{u} \Exp{\beta}{v} =\Exp{\theta}{w}.$$
where $\mathbf{w} = \mathbf{u} \rot{-2\beta}{v}$.
\end{lemma}

\begin{proof} We compute:
$$\begin{array}{lllll}
  \Exp{-\beta}{v}\Exp{\theta}{u} \Exp{\beta}{v} 
 &= &
 \Exp{-\beta}{v}(\cos(\theta) + \sin(\theta) \mathbf{u}) \Exp{\beta}{v} 
 &= &
  \cos(\theta) + \sin(\theta){\Exp{-\beta}{v}} \mathbf{u}\Exp{\beta}{v} \\
  &= &
 \cos(\theta) + \sin(\theta) \mathbf{u}\rot{-2\beta}{v} 
&= & \Exp{\theta}{w},
\end{array}
$$
 where $\mathbf{w} =  \mathbf{u}\rot{-2\beta}{v}.$
\end{proof}

Since $\SO$ acts transitively on $S^2$  from Lemma~\ref{lem:conjugation} we have:
 \begin{corollary}\label{cor:conj}
   The conjugacy class of $\mbf{x} = \Exp{\theta}{u}$ has the form
$$ \mbf{x}^\SU = \{ \Exp{\theta}{v}:  \  \mathbf{v} \in S^2 \}.$$
\end{corollary}

\begin{definition}\label{def:S^2}
{\rm For $ 0 < \psi <2 \pi$ we denote by  $S^2_\psi$ the quandle with underlying set $S^2$ and product 
      $\mathbf{u}*\mathbf{v} =\mathbf{u} \rot{\psi}{v}$,  
for $\mathbf{u},\mathbf{v} \in S^2$. We call this a {\it spherical quandle}.
}
\end{definition}

\begin{lemma} \label{lem:QuandleIsomorphism}
For $0 < \theta < \pi$ the mapping $\mbf{u} \mapsto \Exp{\theta}{u}$ is an isomorphism from quandle $S_{\psi}^2$ with $\psi = 2\pi - 2\theta$ to the conjugacy class $\tilde{C}_\theta = \{ \Exp{\theta}{u}:  \  \mathbf{u} \in S^2 \}$ considered as a quandle under conjugation:
$\mbf{p*q=q^{-1} p q}$. 
\end{lemma}

\begin{proof}
The result follows from Lemma~\ref{lem:conjugation} by taking $\beta = \theta.$
\end{proof}

\begin{lemma} 
$SU(2)$ is a perfect group, that is, it is its own commutator subgroup.
\end{lemma}

\begin{proof}  By \cite{Porteous}, Prop.~10.24  every unit quaternion $\mbf{q}$ has the form
$\mbf{q=aba^{-1}b^{-1}}$ for non-zero quaternions $\mbf{a}$ and $\mbf{b}$. The same holds
if we normalize $\mbf{a}$ and $\mbf{b}$. 
\end{proof}

\begin{lemma}\label{lem:Lambda}
If $\mbf{x}=\Exp{\theta}{u}$ for $0 < \theta < \pi$ then the centralizer $C(\mbf{x})$ is the circle group:
$$\{ \Exp{\beta}{u} : \ 0 \leq \beta < 2\pi \}.$$
Hence, the longitudinal group for $(SU(2),\mbf{x})$ is given by
$$\Lambda =  C(\mbf{x}) \cap SU(2)' = C(\mbf{x}) =  \{ \Exp{\beta}{u} : \ 0 \leq \beta < 2\pi \}.$$
\end{lemma}

\begin{proof} This follows from Lemma~\ref{lem:conjugation} and the fact that for $0<\beta<\pi$,
$ \mathbf{u} \rot{\beta}{\mathbf{v}} = \mathbf{u}$, with $\mathbf{u}, \mathbf{v} \in S^2$ if and only if $\mathbf{v} = \pm \mathbf{u}$ together with the fact that
$$\{ \Exp{\beta}{u} : \ 0 \leq \beta < 2\pi \}=\{ \Exp{\beta}{(-u)} : \ 0 \leq \beta < 2\pi \}.$$ 
\end{proof}

\begin{remark}
{\rm
It is easy to see that for  $\mbf{x}= \Exp{\theta}{u}$  the conjugacy classes $x^\SU = \tilde{C}_\theta$ and $(-\mbf{x})^\SU =\tilde{C}_{\theta + \pi}$ 
 are isomorphic via $ \mathbf{q} \mapsto -\mathbf{q}$ as conjugation quandles. Note also that $\mathbf{u} \mapsto -\mathbf{u}$ leaves the longitude invariant.
Thus for our purposes it suffices to consider only those $\mbf{x} = \Exp{\theta}{u}$ for $ 0 < \theta < \pi$.
Note that $\rot{\psi}{v} =\rot{-\psi}{-v}$. It follows that $S^2_\psi$ is isomorphic to
$S^2_{-\psi}$ via $\mbf{u} \mapsto -\mbf{u}$. Thus when coloring knots by the family of quandles $S^2_\psi$ we
may restrict $\psi$ to the interval $(0,\pi]$. And for the quandles $\tilde{C}_\theta$ we may restrict $\theta$ to
the interval $[\pi/2, \pi)$. 

}
\end{remark}

Fix $\theta \in (0, \pi)$ and $\mbf{x}=\mbf{x}_\theta=\Exp{\theta}{i}$ where $\mathbf{i} = (1,0,0)$  we are interested in computing  $\mathcal{L}_\theta = \mathcal{L}_{\SU}^{\mbf{x}_\theta}.$

\section{Knot Colorings by the Spherical Quandles  $S^2_\psi$} 

Knot group representations in $\SU$ were studied in Klassen~\cite{Klassen}, in particular for all torus knots and twist knots.
We present explicit colorings of torus knots $T(2, n)$ and the figure eight knot  in this section and we compute the longitudinal mappings  of these knots in the next section.

Fix $\psi \in [0, 2 \pi)$ and as above denote by $\rot{\psi}{v}$ the rotation by $\psi$ about ${\mathbf v}$.
Then the quandle structure on $Q=S^2_\psi$ is as defined in Definiton~\ref{def:S^2}, for ${\mathbf u}, {\mathbf v } \in S^2$,  by 
${\mathbf u} *{\mathbf v }={\mathbf  u}  \rot{\psi}{v} $ with right action of the rotation. 
Denote by $\langle {\mathbf u} , {\mathbf v} \rangle$ the inner product of ${\mathbf u}, {\mathbf  v}$ in $\R^3$, so that $S^2_\psi=\{ {\mathbf u } \in \R^3 \ | \ \langle {\mathbf u}, {\mathbf u }\rangle =1 \}$.
We also denote the length of the shortest spherical geodesic segment between ${\mathbf u} , {\mathbf v} \in S^2_\psi$ by $d({\mathbf u}, {\mathbf v} ) = \arccos(\langle {\mathbf  u} , {\mathbf  v } \rangle )$,
and  we denote the (directed) spherical angle at a vertex $\mathbf{v}$  formed by 
 three unit vectors  ${\mathbf u}, {\mathbf  v} , {\mathbf  w } \in S^2_\psi$ 
 by $\angle({\mathbf u} {\mathbf  v } {\mathbf  w} )= \psi$ if  ${\mathbf w} ={\mathbf u}\rot{ \psi}{v}$ and $0 \le \psi < 2\pi$.

Let $\kappa : \{1,\dots, n \} \rightarrow \{0, \dots, n \}$ and $\epsilon : \{1, \dots, n \} \rightarrow \{1,-1\}$  be the  Wirtinger code of a tangle diagram $T$ as described in Definition~\ref{Wirtinger}.
We observe that the coloring condition depicted in Figure~\ref{Xing} is formulated as follows.
Let $Q= S^2_\psi$. Then a coloring  $\rho \in {\rm Col}_Q^{\mbf{x}}(T)$ corresponds to a  sequence of points
$$ (\rho(0), \ldots, \rho(n) ) \in {(S^2_\psi)}^{n+1}$$ satisfying 
\begin{eqnarray*} \label{ColEqns}
\rho(i) = \rho(i-1) \rot{\psi}{\rho(\kappa({\it i}))}^{\epsilon(i)}, \quad i \in \{1,\ldots,n \} .
\end{eqnarray*} 
Thus we have the following, as stated in \cite{Klassen}:

\begin{lemma}\label{lem:color}
For a coloring of a knot diagram by $Q= S^2_\psi$, 
consider a crossing  with the  colors $({\mathbf a}, {\mathbf b})$ as depicted in Figure~\ref{Xing}.
Then  ${\mathbf c}={\mathbf a} *{\mathbf b} $
if and only if  $d({\mathbf a}, {\mathbf b} ) =d ({\mathbf b}, {\mathbf c })$ and 
$\angle({\mathbf a }{\mathbf b} {\mathbf c })=\psi$. 
In particular, any orientation preserving isometry of the sphere takes a coloring to a coloring.
\end{lemma}

\begin{corollary}\label{cor:rotcolor}
For any coloring $\rho \in {\rm Col}_Q^x (T)$ such that $(\rho(0), \ldots, \rho(n) ) \in {(S^2_\psi)}^{n+1}$,
$$ (\rho(0)  \rot{\phi}{ {x} }, \ldots, \rho(n)  \rot{\phi}{ {x} })$$ defines a coloring in $ {\rm Col}_Q^x (T)$  for all $\phi \in [0, 2\pi)$.

\end{corollary}

\begin{remark}
{\rm
As $\psi$ varies, we have a continuous family $\{ S^2_\psi : \psi \in (0, 2 \pi) \}$ of quandles.
This leads to continuous family of knot colorings by $\tilde{C}_\theta$, 
where $\theta=\pi - \psi/2$.
The longitudinal mapping invariant, then, can be seen as a continuous family of invariants $\mathcal{L}^{\mbf{x}_\theta}_{\SU}$ over $\theta$. 

}
\end{remark}

Let $T$ be a tangle corresponding to a 2-bridge knot $K$.
Then we may choose a diagram of $T$ to be a diagram with two bridges, i.e.,  there are two arcs  $x_0$ and $x_1$ 
such that $x_0$ is the initial arc of $T$, and the colors of $x_0$ and $x_1$ uniquely determine a color of all arcs of $T$.

Let $Q=S^2_\psi$, and we fix $\mbf{x}=\mbf{i}=(1,0,0)$.
Thus for all elements $\rho \in  {\rm Col}^{\mbf{x}}_Q(T)$, we have   $\rho (x_0)=\mbf{x}$ as $x_0$ is the initial arc of $T$.
Let $E \subset S^2_\psi$ be half of the equator,
$$ E=\{  ( \cos \phi , \sin \phi, 0) :
0 \leq \phi \leq \pi  \} . $$

\begin{lemma}\label{lem:circles}
Let $Q=S^2_\psi$ and let $\mbf{x} \in \SU$, $T$, $x_0$, and $x_1$ be as above.
Suppose that the number $h$  of elements  $\rho \in  {\rm Col}^x_Q(T)$ such that  $\rho(x_1) \in  E$ and $\rho(x_1) \neq \rho (x_0)$
is finite. 
Then  ${\rm Col}^x_Q(T)$ is homeomorphic to $h$ copies of $S^1$. 
\end{lemma}

\begin{proof}
This follows from  Corollary~\ref{cor:rotcolor}.
\end{proof}

\begin{remark}
{\rm
In \cite{Klassen}, non-abelian representations of knot groups in $\SU$ for torus knots and twist knots up to conjugation action were determined by Klassen.
For each $\psi$, ${\rm Col}^x_Q(T) \cap E$, $Q = S^2_\psi$ corresponds to Klassen's representation.  
Thus the sets  ${\rm Col}^x_Q(T)$ are known from the paper \cite{Klassen}. We determine explicit colorings of $T(2,n)$ and the figure 8 knot by $S_\psi^2$ in the next two subsections and compute the longitudinal mappings for these knots in the next section.}
\end{remark}

\subsection{Colorings of the torus knots $T(2,n)$ by $S^2_{\psi}$}

Let $n=2k+1$ and we label the arcs  of $T(2, n)$ by ${u}_i$ as  in Figure~\ref{t2nu}.
For later convenience in computing the longitude,  we use the notation 
${u}_i={q}_{2i}$  and ${u}_{k+i}={q}_{2i-1}$ for $i=0, \ldots, k$ as depicted in Figure~\ref{t2nu}. Note that the subscripts on the $\mbf{u}$'s correspond to the labeling of the Wirtinger code (Definition~\ref{Wirtinger}).

\begin{figure}[htb] 
\begin{center}
\includegraphics[width=2.5in]{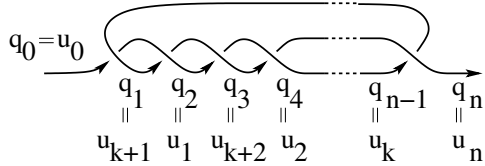}
\end{center}
\caption{Arc labeling  diagram for $T(2,n)$}
\label{t2nu}
\end{figure}

Let $\mbf{p}_i$, $i=0, \ldots, n-1$  (subscripts taken modulo $n$), be a set of points on $S^2$ that are the vertices of a spherical regular $n$-gon
arranged in counterclockwise order,
for example, $$\mbf{p}_i=(\sqrt{1-r^2} \cos ( (2 \pi / n) i ), \sqrt{1-r^2} \sin ( (2 \pi / n) i  ) , r )$$ where $ r \in (-1, 1)$.
 Then the side lengths $d(\mbf{p}_i , \mbf{p}_{i+1})$ and  the angles $\angle{\mbf{p}_{i-1}\mbf{p}_i \mbf{p}_{i+1} }$ are constant.

\begin{lemma} \label{lem:color_criterion}
Let $n=2k+1$.
Let $C_h$ be the map ${\cal A}(T(2,n)) \rightarrow S^2_{\psi}$ defined by  $C_h( {q}_i ) = \mbf{p}_{hi}$ where the subscripts are 
taken modulo $n$. If $\psi=\angle{\mbf{p}_{(i-1)h}\mbf{p}_{ih} \mbf{p}_{(i+1)h} }$, then $C_h$ defines a coloring of $T(2,n)$.
\end{lemma}

\begin{proof}
From Figure~\ref{t2nu}, $C_h( {q}_i)$, $i=0, \ldots, n-1$, gives rise to a non-trivial coloring if the following equations are satisfied:
$C_h ( {q}_{i-1} )*C_h ( {q}_i)= C_h( {q}_{i+1}) $ for all $i$, where the subscripts are taken modulo $n$.
Since the lengths $d( \mbf{p}_{ih} , \mbf{p}_{(i+1)h} )$ and the angles $\angle{\mbf{p}_{(i-1)h}\mbf{p}_{ih} \mbf{p}_{(i+1)h} }$ are constant, 
the conditions for a coloring in Lemma~\ref{lem:color} are satisfied.
\end{proof}

\begin{figure}[h]
\begin{center}
\includegraphics[width=2.1in]{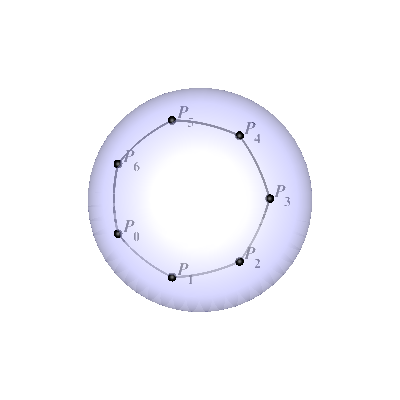}
\includegraphics[width=2.1in]{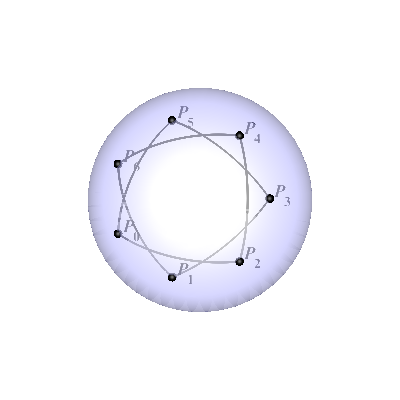}
\includegraphics[width=2.1in]{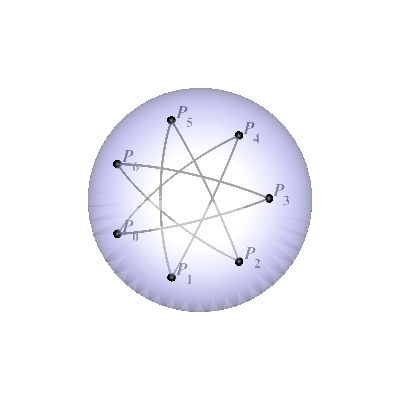}
\end{center}
\caption{Spherical regular star polygons for $n = 7$}
\label{t25regpentagon}
\end{figure}
 
\begin{example}
{\rm
For $n=7$ and $h=1$, $2$, $3$ respectively,
the points corresponding to the colorings
are illustrated in Figure~\ref{t25regpentagon}.
For each $h=1,2,3$, the ranges of $\psi$ are computed from Lemma~\ref{lem:thetavalues} below as
$(5/7)\pi< \psi < (9/7)\pi$ ,
$(3/7)\pi< \psi < (11/7)\pi$, and
$(1/7)\pi< \psi < (13/7)\pi$.

}
\end{example}

\begin{lemma}\label{lem:thetavalues}
Let $n=2k+1$.
For $h=1, \ldots, k$, 
there exists a regular star  $n$-gon with vertices $\mbf{p}_{ih}$, $i=1, \ldots, n-1$, with $\psi=\angle{\mbf{p}_{(i-1)h}\mbf{p}_{ih} \mbf{p}_{(i+1)h} }$  if  and only if $$ (  n-2h  ) \pi /n  < \psi <   (n+ 2h ) \pi / n  . $$
\end{lemma}

\begin{proof}
\begin{sloppypar}
Assume that there exists such a regular star $n$-gon with 
 $\psi=\angle{\mbf{p}_{(i-1)h}\mbf{p}_{ih} \mbf{p}_{(i+1)h} }$.
The angle $\angle{\mbf{p}_{(i-1)h}\mbf{p}_{ih} \mbf{p}_{(i+1)h} }$ is smaller as  the length  $d(\mbf{p}_{ih} , \mbf{p}_{(i+1)h})$ is smaller, 
and hence  the lower bound of such $\psi$ is computed as the corresponding angle $\angle{\mbf{p}_{(i-1)h}\mbf{p}_{ih} \mbf{p}_{(i+1)h} }$
for a planar, infinitesimal  regular $n$-gon formed by $\mbf{p}_{ih}$.
\end{sloppypar}

For the planar regular $n$-gon with vertices $\mbf{p}_i$, $i=0, \ldots, n-1$ in this cyclic order, 
the angle $\angle \mbf{p}_{i-1} \mbf{p}_i \mbf{p}_{i+1}$ equals $ [(n-2)/n] \pi$ since there are $n-2$ triangles in a regular $n$-gon. 
This angle $\angle \mbf{p}_{0-1} \mbf{p}_0 \mbf{p}_{1}$ at $\mbf{p}_0$ 
 is equally divided to the angle $\angle{\mbf{p}_i \mbf{p}_0 \mbf{p}_{i+1}}$ inscribed by $\mbf{p}_i $ and $\mbf{p}_{i+1}$ for each $i$, hence 
 $\angle{\mbf{p}_i \mbf{p}_0 \mbf{p}_{i+1}}= \pi / n$. 
The angle $\angle \mbf{p}_1 \mbf{p}_0 \mbf{p}_h$ and $\angle \mbf{p}_{kh}  \mbf{p}_0 \mbf{p}_{n-1}$ consist of $(h-1)$ parts of $\pi / n$. 
Hence   the lower bound is computed as 
$$\angle \mbf{p}_{h} \mbf{p}_{0} \mbf{p}_{kh}  = \angle \mbf{p}_{n-1} \mbf{p}_0 \mbf{p}_1 - ( \angle \mbf{p}_1 \mbf{p}_0 \mbf{p}_h + \angle \mbf{p}_{kh} \mbf{p}_0 \mbf{p}_{n-1} ) 
= [ \ (n-2) - 2(h-1) \ ]  \pi /n = (  n-2h ) \pi /n .$$
 See Figure~\ref{angles}.
 Since the bounds are symmetric about $\pi$, we obtain the upper bound of
 $$ \pi+ (\pi - (  n-2h ) \pi /n ) =  (n+ 2h  ) \pi / n  $$
 as desired.
\end{proof}

\begin{figure}[htb]
\begin{center}
\includegraphics[width=1.5in]{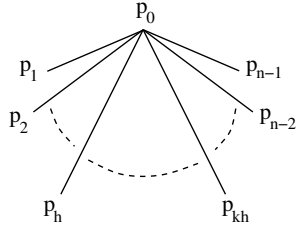}
\end{center}
\caption{Angles of $C_h$}
\label{angles}
\end{figure}

\begin{corollary}\label{cor:T2ncoloring}
For $n = 2k+1$ there is a non-trivial coloring of $T(2,n)$ by $S_\psi^2$ if and only if 
  $$ (  n-2h  ) \pi /n  < \psi <   (n+ 2h ) \pi / n, $$
for some $h = 1, \dots, k$.
\end{corollary}
\begin{proof} Immediate from Lemma~\ref{lem:color_criterion} and Lemma~\ref{lem:thetavalues}.
\end{proof}

\begin{remark}\label{rem:Ino}
{\rm
For fixed  $n$ and $h$ as $\psi$ ranges over the interval
$(\,  (  n-2h ) \pi /n  , \pi \, ]$ continuously,
 the polygons formed by the lengths  $d( \mbf{p}_{ih} , \mbf{p}_{(i+1)h} )$
 continuously change from an infinitesimal polygon to a polygon on the equator.
 As $\psi$ approaches the lower bound $ (  n-2h ) \pi / n $, the polygon converges to a planar polygon.

The coloring condition holds for the Euclidean rotational quandles investigated in \cite{Ino}, in which Inoue  proved 
that there exists a non-trivial coloring by planar rotational quandles if and only if the Alexander polynomial has 
a root on the unit circle $S^1 \subset {\mathbb C}$. The Alexander polynomial of $T(2,n)$ is a 
factor of $x^{2n} - 1$.

}
\end{remark}

\begin{remark}\label{rem:t2nFox}
{\rm
In \cite{Klassen},  $\SU$ representations up to conjugacy are studied.
Furthermore, in \cite{HK}, under certain conditions  satisfied by $T(2,n)$ and twist knots, 
the representations are deformations of  dihedral representations at $\psi = \pi$. 

These results are seen in the above continuous family of star polygons. 
They start from infinitesimal planar polygons and converge to the equatorial ``polygons'' that correspond to
Fox colorings by dihedral quandles. 

}
\end{remark}

\begin{proposition}
Let $Q=S^2_\psi$ and $T$ be a tangle of $T(2,n)$ as depicted in Figure~\ref{t2nu}. 
For $n = 2k+1$ and $h=1, \ldots, k$, if  $ (  n-2h  ) \pi /n  < \psi \leq (  n-2h +2 ) \pi /n $ then 
$${\rm Col}^{\mbf{x}}_Q(T)=\sqcup_{h} S^1,$$
  $h$ copies of disjoint circles.
\end{proposition}

\begin{proof}
By Lemma~\ref{lem:thetavalues}, if $\psi$ is in the stated range, then for any $h' \leq h$, $h'$ satisfies the condition stated 
in Lemma~\ref{lem:thetavalues}.
In Figure~\ref{t2nu}, the arcs ${q}_0$ and ${q}_1$ are taken as $x_0$ and $x_1$ in Lemma~\ref{lem:circles}.
Hence in the notation in Lemma~\ref{lem:circles}, ${\rm Col}^x_Q(T) \cap E$ consists of $h$ points, and 
the result follows from Lemma~\ref{lem:circles}.
\end{proof}

\subsection{Colorings of the figure eight knot by $S_\psi^2$}

In this subsection we describe the colorings of a figure eight knot by the spherical quandle $S^2_\psi$.

\begin{lemma}\label{lem:fig8color}
A sequence $U=(\mbf{u}_0, \mbf{u}_1, \mbf{u}_2, \mbf{u}_3)$ defines a coloring if and only if the following conditions are satisfied in $S^2_\psi$:
$ d( \mbf{u}_1 , \mbf{u}_2)=   d( \mbf{u}_2 , \mbf{u}_0 )=  d( \mbf{u}_0 , \mbf{u}_3), \   d( \mbf{u}_0 , \mbf{u}_1) =  d( \mbf{u}_1 , \mbf{u}_3)=  d( \mbf{u}_3 , \mbf{u}_2) , $ and 
$\angle(\mbf{u}_0 \mbf{u}_2 \mbf{u}_1)=\angle(\mbf{u}_0 \mbf{u}_1 \mbf{u}_3)=\angle(\mbf{u}_2 \mbf{u}_3 \mbf{u}_1)=\angle(\mbf{u}_2 \mbf{u}_0 \mbf{u}_3)=\psi$. 
\end{lemma}

\begin{proof}
Direct inspection of Figure~\ref{figeight} gives the following:
$$ \mbf{u}_0 * \mbf{u}_2 = \mbf{u}_1, \ 
\mbf{u}_0 * \mbf{u}_1 = \mbf{u}_3, \ 
\mbf{u}_2 * \mbf{u}_3 = \mbf{u}_1, \ 
\mbf{u}_2 * \mbf{u}_0=\mbf{u}_3, $$
where $\mbf{u}_4=\mbf{u}_0$ and the equalities are derived from the crossings. 
By Lemma~\ref{lem:color} the statement follows.
\end{proof}

\begin{figure}[h]
\begin{center}
\includegraphics[width=2.1in]{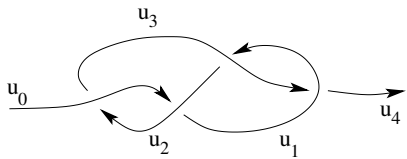}
\end{center}
\caption{Colorings of the figure eight knot }
\label{figeight}
\end{figure}

\begin{lemma}\label{lem:fig8sol}
For $\psi = 2 \pi / 3$ and $\psi = 4 \pi / 3$ there is a unique solution $U$  to the equations in Lemma~\ref{lem:fig8color}
such that $$\mbf{u}_0=\mbf{x_\psi}=(1,0,0) = \mbf{i} \text{\ and \ } \mbf{u}_2 = (cos(\beta), sin(\beta),0)  .$$ The solution $U$ forms a regular spherical tetrahedron.  In this  case  $\beta = \arccos \left( - 1/3 \right).$

For $2 \pi / 3 <  \psi < 4\pi/3$, there are two nontrivial solutions $U$ to the equations in Lemma~\ref{lem:fig8color}
such that $$\mbf{u}_0=\mbf{x_\psi}=(1,0,0) = \mbf{i} \text{\ and \ } \mbf{u}_2 = (cos(\beta), sin(\beta),0)  .$$ 
The solutions are determined by the two values of $\mbf{u}_2 $, 
$\mbf{u}_2= ( \cos(\beta_i),  \sin(\beta_i), 0)$, where for $i=1,2$,
\begin{eqnarray*}
\beta_1 &=&  \pi - \arccos( \, - 1+\sqrt{4\, \cos^2(\psi)-4\, \cos(\psi)-3 )} / 2\,  (\cos (\psi) -1 ), \\
\beta_2&= & \arccos(\,    1+\sqrt{4\, \cos^2(\psi)-4\, \cos (\psi) -3 )} / 2\,  (\cos (\psi) -1 ). 
\end{eqnarray*}

\end{lemma}

\begin{proof} This comes directly from {\it Maple} computations. The Maple worksheets can be found at \cite{Maple}.
\end{proof} 

 \begin{remark} {\rm Note that by Lemma~\ref{cor:rotcolor} it suffices to restrict $\beta$ to the interval $(0,\pi]$.}
\end{remark} 

\begin{remark}
{\rm
{\it Maple} computations give the above exact solutions. 
It was also pointed out by Shin Satoh (via personal communication) that the spherical laws of  sine and cosine, together with
the area formula that a spherical triangle with angles $\alpha, \beta, \gamma$ has area 
$\alpha + \beta + \gamma - \pi$, yield the  solutions.
}
\end{remark}

\begin{figure}[h]
\begin{center}
\includegraphics[width=2.5in]{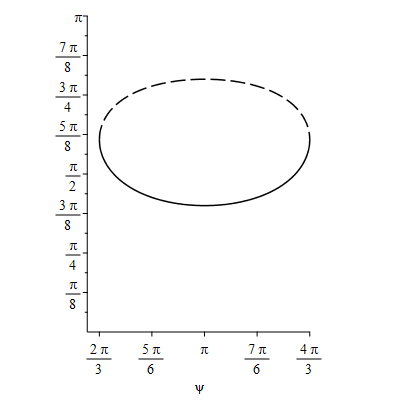} 
\end{center}
\caption{The graphs of $\beta_i$, $i=1,2$,   representing colorings of the figure eight knot }
\label{fig8graph}
\end{figure}

\begin{remark}
{\rm
The solutions for $\beta_i$ for $i=1,2$  in Lemma~\ref{lem:fig8sol}
are  plotted in Figure~\ref{fig8graph} for $\psi \in [2 \pi / 3, 4\pi/3]$. Each angle $\beta_i$  is 0 outside of this interval. Hence the colorings are trivial for $\psi$ outside this interval.
}
\end{remark}

\begin{figure}[h]
\begin{center}
\includegraphics[width=2.1in]{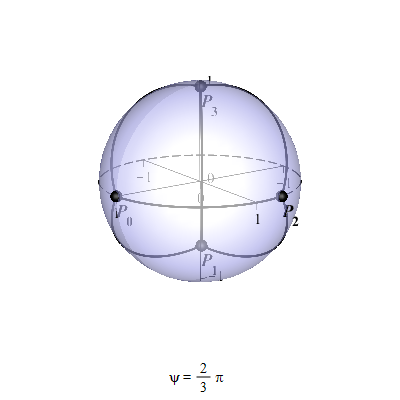}
\includegraphics[width=2.1in]{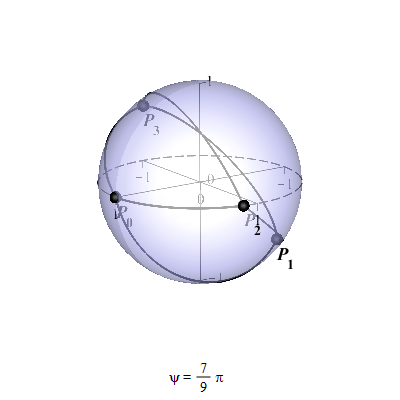}

\includegraphics[width=2.1in]{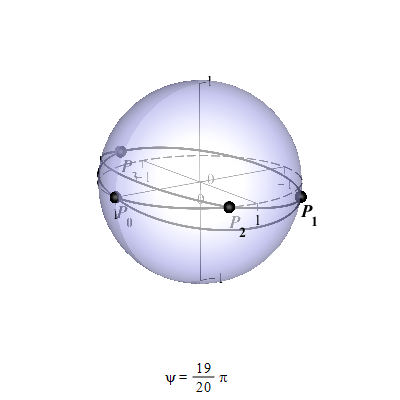}
\includegraphics[width=2.1in]{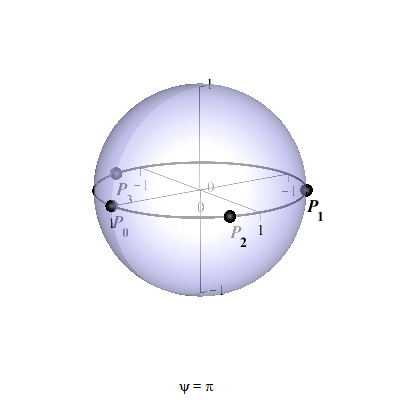}
\end{center}

\caption{Colorings for the figure eight knot by $S_\psi^2$ for $\psi=2\pi / 3, 7\pi / 9$, $19 \pi / 20$ and  $\pi.$}
\label{fig8}
\end{figure}

\begin{remark}
{\rm
The solutions $U$  in Lemma~\ref{lem:fig8sol} for $\psi=2\pi / 3, 7\pi / 9, 19 \pi / 20$ and $ \pi $
 form vertices of  spherical tetrahedra as depicted in Figure~\ref{fig8}.

We recall that the figure eight knot is non-trivially colorable by the tetrahedral quandle 
(the solution $U$ at $\psi=2 \pi/3$) and the dihedral quandle $R_5$ (Fox 5-colorable).
Note also that since the minimal diagram in Figure~\ref{figeight} has only four arcs, four colors in $R_5$ are used for
non-trivial colorings. Up to mirror symmetry, there are two choices of elements of $\mbf{u}_2$ from $R_5$  for a fixed 
element for $\mbf{u}_0$. 
As in Remark~\ref{rem:t2nFox}, there are continuous family of solutions as $\psi$ varies from $2 \pi /3$ to $\pi$.
A single regular tetrahedral coloring bifurcates to two branches of solutions as in Lemma~\ref{lem:fig8sol},
and converges to the two solutions of Fox colorings, as described in \cite{HK}.
Animations of this situation can be found at \url{http://shell.cas.usf.edu/~saito/SphericalQuandle/}.
}
\end{remark}

\begin{remark}
{\rm
More generally, Klassen~\cite{Klassen} described the representations of knot groups  in $\SU$ 
for twist knots ${\rm Tw}_m$, $m>0$, and proved that up to conjugation it consists of $m/2$ circles if $m$ is even,
and $\lfloor m/2 \rfloor$ circles and a single open arc if $m$ is odd.
The cases $m=1$ and $m=2$ correspond to the trefoil and the figure eight knot, respectively.
}
\end{remark}

\begin{remark}
{\rm
It is well known that the Alexander polynomial of ${\rm Tw}_m$ for odd $m$ is given by
$\Delta_{{\rm Tw}_m}(t)=(m+1) t^2 - 2mt + (m+1)$. 
Direct calculations show that $\Delta_{{\rm Tw}_m}(t)$ has roots on $S^1 \subset {\mathbb C}$, and by \cite{Ino}, 
there is a nontrivial coloring by planar rotational quandle for 
$\psi={\rm arg} (\alpha)$, where $\alpha$ is its root. Let $\alpha$ be the root with smaller argument.
Then  for odd $m$ there is a non-trivial coloring of ${\rm Tw}_m$ by $S^2_\psi$ for 
${\rm arg} (\alpha) < \psi < {\rm arg} (\bar{\alpha}) $.
}
\end{remark}

\section{Longitudinal Mapping Invariant Values } 
In this section we determine the invariant values 
$ \mathcal{L}_\theta$ for the torus knots $T(2,n)$ and the figure eight knot.

\subsection{Torus knots $T(2,n)$}

We used the labeling of the diagram of $T(2,n)$ in Figure~\ref{t2nu}, where 
$n=2k+1$ is odd. 

\begin{lemma}for $n = 2k+1$ and $h = 1, \dots, k$, $T(2,n)$ is non-trivially colored by  $\tilde{C}_\theta$ if and only if
$${\frac { \left( n-2h \right) \pi}{2n}} < \theta <
{\frac { \left( n+2h \right) \pi}{2n}}.$$
\end{lemma}
\begin{proof}
By Lemma~\ref{lem:QuandleIsomorphism} for $0 < \theta < \pi$ the quandle $S_{\psi}^2$, $\psi =2\pi -2\theta$, is isomorphic to the conjugacy class $\tilde{C}_\theta = \{ \Exp{\theta}{u}:  \  \mathbf{u} \in S^2 \}$ considered as a quandle under conjugation:
$\mbf{p*q=q^{-1} p q}$. Clearly the isomorphism $\mbf{u} \mapsto \Exp{\theta}{u}$ takes a coloring to a coloring. By Corollary~\ref {cor:T2ncoloring}
for $n = 2k+1$ there is a non-trivial coloring of $T(2,n)$ by $S_\psi^2$ if and only if for some $h = 1, \dots, k$ we have
  $$ (  n-2h  ) \pi /n  < \psi <   (n+ 2h ) \pi / n, $$
since   $\psi =2\pi -2\theta$ this is equivalent to
$${\frac { \left( n-2h \right) \pi}{2n}} < \theta <
{\frac { \left( n+2h \right) \pi}{2n}}.$$

\end{proof}

\begin{lemma}~\label{lem:toruscolor}
Let $n=2k+1$, $k\geq 1$. 
Let $G$ be a group. 
Let  $q_i$, $i=0, 1, \ldots, n-1$,  be the colors of the arcs, as  depicted  in Figure~\ref{t2nu},
of a coloring of the diagram by $G$.
Then $q_i $ satisfy 
$q_{i+1}=q_{i}^{-1} q_{i-1} q_{i}$ for $i=1, \ldots, n-1$ and 
$q_{n}=q_0$.
\end{lemma}

We thank Razvan Teodorescu for the idea of the following proof.

\begin{lemma}\label{lem:longitude}
Let  $n=2k+1$, $k\geq 1$, 
and $G$ be a group. 
For a coloring $C$ of the diagram of $T(2,n)$ in Lemma~\ref{lem:toruscolor},
let $q=q_0 q_1$. 
Then 
the longitude is given by
$\mathcal{L}(C) = q_0^{-2n} q^n$. 
\end{lemma}

\begin{proof}
By Lemma~\ref{lem:toruscolor},
we have $q_i q_{i+1} = q_{i+1} q_{i+2} $ for $i=0, \ldots, n-2$, and $q_{n-1} q_0=q_0 q_1$. 
Note that $q= q_i q_{i+1} $ for all $i$.

For any coloring $C$, 
from Figure~\ref{t2nu}, we compute the longitude as 
$$\mathcal{L}(C) = q_0^{-n} \ ( q_1 q_3 \cdots q_{2k -3} ) \ ( q_0 q_2 \cdots q_{2k} ) .$$
To evaluate this, we compute 
 $$q_0^{2n} \mathcal{L}(C) = q_0^{n} \ ( q_1 q_3 \cdots q_{2k -3} ) \ ( q_0 q_2 \cdots q_{2k} ) .$$
 Since $q_0 q_1= q_1 q_2$, we have
\begin{eqnarray*}
q_0^{2n}  \mathcal{L}(C) &=& (q_0 \cdots q_0) \ ( q_0 q_1 ) \ ( q_3 \cdots q_{2k -3} ) \ ( q_0 q_2 \cdots q_{2k} ) \\
&=&  (q_0 \cdots q_0)  \ ( q_1 q_2 ) \  ( q_3 \cdots q_{2k -3} ) \ ( q_0 q_2 \cdots q_{2k} ) .
\end{eqnarray*}
Further applying $q_0 q_1= q_1 q_2$ and  $q_2 q_3= q_3 q_4$, we obtain
\begin{eqnarray*}
 &=& (q_0 \cdots q_0) \ ( q_0 q_1 ) \ ( q_2 q_3 ) \ ( q_5 \cdots q_{2k -3} ) \ ( q_0 q_2 \cdots q_{2k} ) \\
&=&  (q_0 \cdots q_0)  \ ( q_1 q_2 )\  (  q_3 q_4 ) \  ( q_5 \cdots q_{2k -3} ) \ ( q_0 q_2 \cdots q_{2k} ) .
\end{eqnarray*}
Inductively we obtain 
$$ (q_0 \cdots q_0)  \ ( q_1 q_2   q_3 q_4  \cdots q_{2k} ) \ ( q_0 q_2 \cdots q_{2k} ) .$$
There are $k+1$ copies of $q_0$ in the first factor, $(q_i)_{i=1}^{2k}$ in the second factor, and consecutive even terms in the third factor.
Then we continue with 
\begin{eqnarray*}
 &=& (q_0 \cdots q_0)\  ( q_1 q_2   q_3 q_4  \cdots q_{2k-1} )\ (q_{2k} q_0)  \ ( q_2 \cdots q_{2k} ) \\
 &=& (q_0 \cdots q_0)\  ( q_1 q_2   q_3 q_4  \cdots q_{2k-1} )\ (q_0 q_1 ) \ ( q_2 \cdots q_{2k} ) \\
 &=& (q_0 \cdots q_0)\  ( q_1 q_2   q_3 q_4  \cdots q_{2k-1} )\ (q_0 q_1 q_2 ) \ ( q_3 \cdots q_{2k} ) \cdots.
  \end{eqnarray*}
  In the last line, the left consecutive sequence keeps shifting to the left, as the middle pair  $ (q_0 q_1 ) $ shifts to the left.
  Inductively, we obtain 
$  q_0^{2n}  \mathcal{L}(C) = ( \prod_{i=0}^{n-1} q_i )^2 = q^n$. 
Hence we obtain $ \mathcal{L}(C) = q_0^{-2n} q^n $. 
\end{proof}

\begin{figure}[htb]
\begin{center}
\includegraphics[width=3.2in]{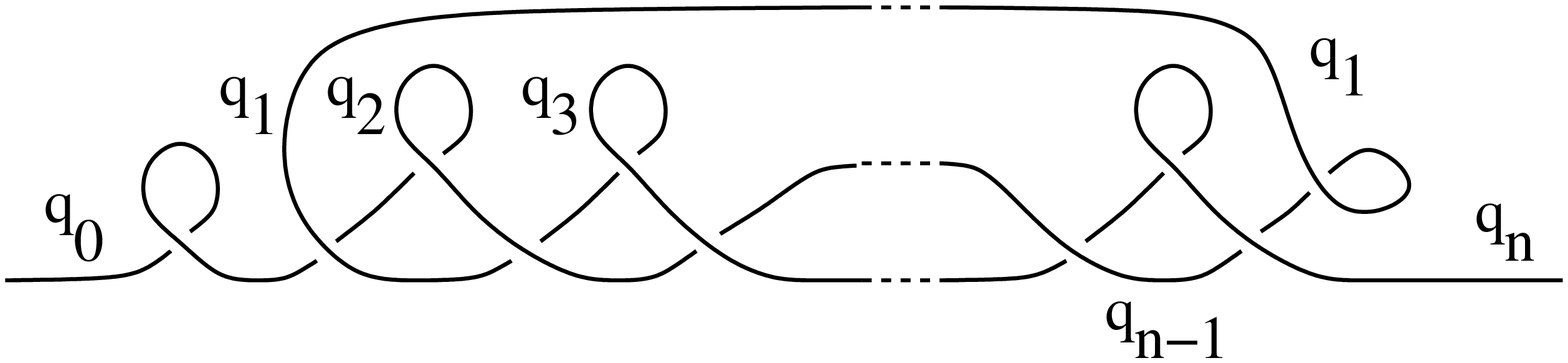}
\end{center}
\caption{Colored diagram for $T(2,n)$ with loops}
\label{t2nloops}
\end{figure}

\begin{remark}
{\rm
\begin{sloppypar}
In the proof of Lemma~\ref{lem:longitude}, once the computation of 
$ \mathcal{L}(C) =  q_0^{-2n}   ( \prod_{i=0}^{n-1} q_i )^2 $ is obtained, we found a  diagrammatic method of
obtaining the same formula. Specifically, from the diagram in Figure~\ref{t2nloops}, we can read off the longitude directly
as $q_0^{2n} \mathcal{L}(C) =  ( \prod_{i=0}^{n-1} q_i )^2 $.
\end{sloppypar}
}
\end{remark}

It is noteworthy that in the following theorem, the longitudinal mapping depends only on $\theta$, and not 
on the different  colorings $C$ corresponding to $\theta$. 

\begin{theorem}\label{thm:t2n}
For any non-trivial coloring $C$ of $T(2,n)$, the value of the longitudinal mapping  for $(\SU, \mbf{x})$ where $\mbf{x} =\Exp{\theta}{i}$ is given by
$$\mathcal{L}(C) = \Exp{(\pi-2 n \theta)}{i} =-\cos ( 2n \theta) + \sin(2 n \theta)\mathbf{i}.$$
and for the mirror image $m(T(2,n))$ the value of the longitudinal mapping is given by
$$\mathcal{L}(C) = \Exp{(2n \theta-\pi )}{i} =-\cos ( 2n \theta) - \sin(2 n \theta)\mathbf{i}.$$ 
\end{theorem}

\begin{proof} 
In the case of $G=\SU$ in Lemma~\ref{lem:longitude}, we show that $\mbf{q}^{n}=-1$, where 
$\mbf{q}=\mbf{q}_0 \mbf{q}_1= \mbf{q}_i \mbf{q}_{i+1}$ for all $i$. 
Since 
$$\mbf{q}^{-1} \mbf{q}_i \mbf{q} = (\mbf{q}_{i} \mbf{q}_{i+1})^{-1} \mbf{q}_i (\mbf{q}_i \mbf{q}_{i+1} ) = \mbf{q}_{i+2}, $$ we have 
$\mbf{q}^{-n} \mbf{q}_i \mbf{q}^n = \mbf{q}_i$ for every $i$. 
Then $\mbf{q}^n$ is in $C(\mbf{q}_i)$ for every $i$.
For a non-trivial coloring, there are at least two $ \mbf{q}_i $ and $ \mbf{q}_j$ that do not commute, 
hence by Lemma~\ref{lem:Lambda}, $C( \mbf{q}_i ) \cap C( \mbf{q}_j) = \{ \pm 1\}$,
so that
$\mbf{q}^n= \pm 1$. 

For each $\theta$, we have  $\mbf{q}^n=\pm 1$, and $\mbf{q}^n$ is continuous with respect to $\theta$. 
By Corollary~\ref{cor:conj}, for $\theta=\pi/2$, we have 
$S_{\pi}^2$ isomorphic to the conjugacy class $\tilde{C}_{\pi/2} $.
In this case, the colorings by $S_{\pi}^2$ up to the action of rotations about $\mbf{x}$ 
(cf. Corollary~\ref{cor:rotcolor})
are equivalent to Fox colorings by a dihedral quandle $R_m$ for some $m$. 
In \cite{Klassen},  it was shown that the non-abelian representations of knot groups of torus knots $T(r,s)$ up to conjugacy consist of $(r-1)(s-1)/2$ open arcs. In our case the result implies that 
the set of non-trivial colorings of $T(2, n)$ consists of $n-1$ open arcs each of which contains 
a coloring by the dihedral quandle $R_n$. 
Hence the fact $\mbf{q}^n=-1$ follows if it is proved for colorings by $\tilde{C}_{\pi/2}$. 

Let $\theta=\pi/2$, then $\mbf{q}_0= \Exp{\frac{\pi}{2}}{i} = {\mathbf i}$. In this case $\mbf{q}_1=\cos (2 \pi m / n) {\mathbf i} + \sin (2 \pi m / n) {\mathbf j}$
for some $m$. 
Then we compute 
$$\mbf{q}= \mbf{q}_0 \mbf{q}_1 = - \cos (2 \pi m / n)  + \sin (2 \pi m / n) {\mathbf k} = \Exp{(\pi - 2 \pi m / n) }{\mathbf k}.$$
Hence we obtain 
$$\mbf{q}^n= \Exp{ ( \pi - 2 \pi m / n) n } {\mathbf k} =  \Exp{\pi (n - 2 m ) } {\mathbf k} = -1$$
since $n$ is odd, as desired. 

The resullt for $m(T(2,n))$ follows immediately from the result for $T(2,n)$ via Proposition~\ref{mirror} and the known fact that $r(T(2,n)) = T(2,n)$. 
\end{proof}

\subsection{Figure eight knot} 

The following Lemma is immediate from Lemma~\ref{lem:fig8sol} and the fact that $S_\psi^2$  is isomorphic to $ \tilde{C}_\theta$ when  $\psi = 2\pi-2\theta$ and the fact that the isomorphism $\mbf{u} \mapsto \Exp{\theta}{u}$ takes a coloring to a coloring.

\begin{lemma}  The figure 8 knot  is non-trivially colored by  $\tilde{C}_\theta$ if and only if
$$\frac{\pi}{3} \le \theta \le \frac{2\pi}{3}.$$ 
In which case there are two solutions for each  $\theta \in ( \frac{\pi}{3} , \frac{2\pi}{3})$,
corresponding to the values of $\beta_1$ and $\beta_2$ in Lemma~\ref{lem:fig8sol}.
The colorings for $\theta =  \frac{\pi}{3}$ and $\theta = \frac{2\pi}{3}$ are the same.
\end{lemma}

Let $C(i) = \mbf{u}_i$ be a coloring for the figure 8 knot by $\tilde{C}_\theta$ for  $\theta \in [ \pi / 3, 2 \pi / 3 ]$,
 as shown in Figure~\ref{figeight}. Then from the definition of the longitude 
  we obtain the following.

\begin{lemma}
$\mathcal{L}_\theta(C) = \mbf{u}_2 \mbf{u}_3^{-1} \mbf{u}_0 \mbf{u}_1^{-1}$. 
\end{lemma}

{\it Maple} computations give the following. 

\begin{proposition} If $C$ is a coloring of the figure 8 knot by $\tilde{C}_\theta$  then 
$${\mathcal L}_\theta(C) = (\cos \left( 4\,\theta \right) -\cos \left( 2\,\theta \right) -1)
\pm \sqrt {-1+2\,\cos \left( 4\,\theta \right) -4\,\cos \left( 2\,\theta
 \right) } \, ( \sin \left( 2\,\theta \right) ) \, {\bf i}.$$ The sign $\pm$ depends on the choice of $\beta_i$, $i = 1,2,$ in Lemma~\ref{lem:fig8sol}.

\end{proposition}

The longitude  ${\mathcal L}_\theta (C) $ may be  written as $ {\bf e}^{\phi {\bf i} } $
where $\phi$ is given in terms of the two argument  $\arctan$ by
$$ \phi = \arctan \left( \pm \sqrt {4\, \left( \cos \left( 2\,\theta \right) 
 \right) ^{2}-4\,\cos \left( 2\,\theta \right) -3} \, ( \sin \left( 2\,
\theta \right) ) \, ,2\, \left( \cos \left( 2\,\theta \right)  \right) ^{2}
-\cos \left( 2\,\theta \right) -2 \right) $$
The graph of $\phi$ as a function of $\theta$ is given in Figure~\ref{Longitude1_Longitude2}.

\begin{figure}[htb]
\begin{center}
\includegraphics[width=2.8in]{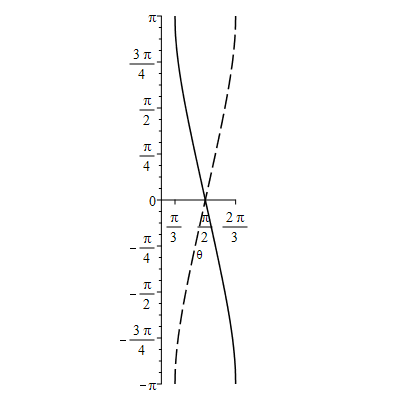}
\end{center}
\caption{The graph of $\phi$ where  ${\mathcal L}_\theta (C) = {\bf e}^{\phi {\bf i} }$ for the figure 8 knot.}
\label{Longitude1_Longitude2}
\end{figure}

\section{Concluding Remarks}

In this paper, the knot coloring polynomial defined by  Eisermann~\cite{Eis-colpoly} with finite quandles 
is generalized to topological quandles as the longitudinal mapping invariant of long knots, which in turn can be 
thought of as a generalization of the quandle 2-cocycle invariant defined in \cite{CJKLS} for finite quandles.
Such generalizations for topological quandles have long been called for, and we propose one in this paper.
The invariant values are concretely  evaluated for torus knots of closed 2-braids $T(2,n)$ and the figure eight knot.

The following questions, for example,  remain to be investigated:
determine the coloring spaces for other knots, in particular knots with more than 2 bridges;
determine the $\theta$-values with non-trivial colorings; 
determine  the invariant values;
relations to other invariants;
investigate continuous cohomology theories of topological quandles, and relate it to the invariant 
discussed in this paper.

\section*{APPENDICES}

\appendix

\section{Eisermann quandles  and generalized Alexander quandles}\label{sec:GAlex}

For an alternative description of  
the invariant $\mathcal{L}$, we focus on  the following  
quandles found in Lemma~25 and Remark~27 of  \cite{Eis-unknot}.

\begin{definition} 
{\rm
Let $G$ be a group and $x \in G$ such that  conjugacy class $x^G$  generates $G$. The conjugacy class $x^G$ is a quandle under
conjugation $a*b = b^{-1}ab$ and $a\ \bar{*}\ b = bab^{-1}$. Let $G'$ be the commutator subgroup of $G$. 
Define the set
$$\Eis(G,x) = \{ (a, g) \in x^G \times G' \  | \ a=x^g \}.$$ 
This set becomes an indecomposable quandle under the operations
$$(a, g)*(b,h)=(a*b, x^{-1}gb), \quad (a, g)\ \bar{*}\ (b,h)=(a\ \bar{*}\ b, xgb^{-1}),$$
We call this the {\rm Eisermann quandle} given by the pair $(G,x)$.  We write
 $$p: \Eis(G,x) \rightarrow x^G, \quad (a,g) \mapsto a,$$
for the projection onto $x^G$. 
}
\end{definition}

Eisermann \cite{Eis-unknot} wrote $\tilde{Q}(G,x)$ for what we call here $\Eis(G,x)$.
Furthermore as he pointed out that this definition is tailor-made to capture the longitude information we need for the proof of Lemma~\ref{lem:long}.

\begin{lemma}\label{lem:Eis} If $G$ is a  group that is generated by the conjugacy class $x^G$ then $x^{G'} = x^G$, 
$\Eis(G,x)$ is an indecomposable quandle and the projection $$p: \Eis(G,x) \rightarrow x^G, \quad (a,g) \mapsto a$$ is a  quandle
epimorphism that is equivalent to $${ \rm inn} : \Eis(G,x) \rightarrow {\rm inn}(\Eis(G,x)).$$
The fiber $p^{-1}(x)$ is $C(x) \cap G'$ where $C(x)$ is the centralizer of $x$ in $G$.
If $C(x) \cap G'$ is abelian then $p: \Eis(G,x) \rightarrow x^G$ is an abelian extension.
\end{lemma} 

\begin{proof} See Lemma 25 in \cite{Eis-unknot} and Appendix B in \cite{CDS}.
\end{proof}

As noted  by Eisermann,  $\Eis(G,x)$ has an alternative description as a generalized Alexander quandle
$\GA(G',f_x)$ where $f_x$ is the inner automorphism $f_x(g) = x^{-1}gx$, $g \in G$. Since $G'$ is a normal
subgroup,  $f_x$ is an automorphism of $G'$ and so $\GA(G',f_x)$ is well-defined.

\begin{lemma}\label{lem:eis}
 For $x$ an element of a group $G$ the quandles $\Eis(G,x)$ and $\GA(G',f_x)$ are
isomorphic. 
\end{lemma}
\begin{proof} It is easy to check that the mapping $:(a,g) \mapsto g$ is the desired isomorphism.
\end{proof}

\begin{remark}  
{\rm 
The Eisermann quandle $\Eis(G,x)$ does not  determine $G$ since there are many
groups in general with the same commutator subgroup.  On the other hand every indecomposable generalized
Alexander quandle $Q =\GA(G,f)$ determines  the group $G$, namely $G ={\rm Inn}(Q)'$,  and determines the automorphism $f \in {\rm Aut}(G)$ up to conjugacy
in ${\rm Aut}(G)$.  Moreover if $Q=\GA(G,f)$ is indecomposable and $e \in {\rm Inn}(Q)$ then $Q \cong \Eis({\rm Inn}(Q), R_e)$
 as noted in Corollary B.3 of \cite{CDS}.
 }
\end{remark}

\section{Alternative interpretation of  $\mathcal{L}$ for Eisermann and Alexander quandles}

We recall  the following two lemmas.
 
\begin{lemma}[Eisermann~\cite{Eis-unknot}, Theorem 30] \label{lem:cover}
Let $p:\tilde{Q}  \rightarrow Q$ be a covering such that $p(q) = x$, $q\in \tilde{Q}$,  and let $T$ be a 1-tangle diagram. Then  the mapping $\tilde{C} \mapsto p\, \tilde{C}$ is a bijection from
 ${\rm Col}_{\tilde{Q}}^q(T)$ to ${\rm Col}_Q^x(T)$. 
\end{lemma}

\begin{lemma}[\cite{CDS} ]\label{lem:end}
Let  $C:  {\cal A} (T) \rightarrow Y$ be a coloring of a  $1$-tangle diagram $T$ by a quandle $X$. 
For the initial and terminal  arcs $0$ and $n$ of $T$, respectively, 
let $x_0=C(0)$ and $x_1=C(n)$.
Then  ${\rm inn}(x_0)=R_{x_0}=R_{x_1}={\rm inn}(x_1)$. 
\end{lemma}

Now let $\tilde{Q} = \Eis(G,x)$ and $Q = x^G$ and $p:\tilde{Q} \rightarrow Q$ as in Lemma~\ref{lem:Eis}, so that $p(x,1) = x$.  Let $\tilde{C} \in Col_{\tilde{Q}}^{(x,1)}$ and 
$C = p\, \tilde{C}$ as in Lemma~\ref{lem:cover}.  Let $\mathcal{L}(C)$ be as defined above.

\begin{proposition}\label{lem:long} In the notation above let  $\tilde{C}$ be the unique lifting of the coloring $C \in {\rm Col}_Q^x(T)$ to ${\rm Col}_{\tilde{Q}}^{(x,1)}(T)$. Then  $\tilde{C}(n) = (x,\mathcal{L}(C)).$
\end{proposition}

\begin{proof} 
Let $w(i)=\sum_{h=1}^i \epsilon(h)$ be the writhe counted along the tangle from 
the initial arc $0$  
along the tangle up until one reaches at the arc $ i$.  By Lemma~\ref{lem:cover} we know that the coloring $C \in Col_Q^x(T)$  lifts to a unique coloring $\tilde{C} \in Col_{\tilde{Q}}^{(x,1)}(T)$.
Write $u_i = C(i)$ for $i = 0, \dots,n$.
Thus we have $\tilde{C}(i) = (u_i,g_i)$ for  $i = 0, \dots,n$.
By Lemma~\ref{lem:end}, we have $u_n=x$.
Assume inductively that $g_i =x^{- w(i) } \prod_{h=1}^i u_{\kappa(h)}^{ \epsilon (h) } $.
One computes using ${*} = *^1$ and $\bar{*} = *^{-1}$:

\begin{eqnarray*}
& & \tilde{C}( i+1 )   \\
&= & (u_{i+1},g_{i+1}  )\\
& =& (u_i,g_i) *^{\epsilon(i+1)} (u_{\kappa(i+1)},g_{\kappa(i+1)} )\\
&= &  (u_{i+1},x^{-\epsilon(i+1)} g_i u_{\kappa(i+1)}^{\epsilon(i+1)} )\\
&= & (u_{i+1},x^{-\epsilon(i+1)} x^{- w(i) }\  (\prod_{h=1}^i u_{\kappa(h)}^{ \epsilon (h) } )\ u_{\kappa(i+1)}^{\epsilon(i+1)}\  ) \\
&=& (u_{i+1},g_{i+1}) .
\end{eqnarray*} 
Taking $i = n$ we see that the  Proposition holds.
\end{proof}

\begin{theorem} In the notation above  and  let $\bar{C}$ be the unique lifting of the coloring  $C \in {\rm Col}_Q^x(T)$ to ${\rm Col}_{{\rm GAlex}(G',f_x)}^1(T)$. Then $\bar{C}(n) = \mathcal{L}(C).$
\end{theorem}

\begin{remark}
{\rm 
Each element  of $\SU$ is a commutator  (\cite{Porteous} Prop.~10.24\, )
so $\SU$ is equal to its own commutator subgroup. Since $\SO$ is a simple group (\cite{Artin}\ ), and the center of $\SU$ is  $\{1,-1\}$ it follows that if 
$\mbf{x} \neq \pm 1$ then the conjugacy class $\mbf{x}^{\SU}$ generates $\SU$. Thus given any $\mbf{x} \in \SU$ with $\mbf{x} \neq \pm 1$ we may apply the results of Appendix~\ref{sec:GAlex} to $(G,x) = (\SU, \mbf{x})$.
}
\end{remark}

\subsection*{Acknowledgements}
We thank Shin Satoh and Razvan Teodorescu for valuable comments.
MS was partially supported by
NIH R01GM109459.


\begin{thebibliography}{99}
\setlength{\itemsep}{-3pt}


\bibitem{Artin}
Artin, E., {\it Geometric Algebra}, Wiley Classics Library Edition, 1988.

\bibitem{BZ}
Burde, B.; Zieschang, H., {\it Knots}, de Gruyter Studies in Mathematics, vol. 5, Walter 
de Gruyter and Co. Berlin, 1985. 


\bibitem{CJKLS} Carter, J.S.; Jelsovsky, D.; Kamada, S.; Langford, L.; Saito, M.,
{\it Quandle cohomology and state-sum invariants of knotted curves and surfaces,}
Trans. Amer. Math. Soc., 355 (2003) 3947--3989.

\bibitem{CKS}
Carter, J.S.; Kamada, S.;  Saito, M.,
Surfaces in $4$-space,
Encyclopaedia of Mathematical Sciences, Vol. 142,
Springer Verlag, 2004.

\bibitem{CDS}
Clark, W.E.; Dunning, L.A.; Saito, M., 
{\it Quandle 2-cocycle knot invariants without explicit 2-cocycles}, Journal of Knot Theory and Its Ramifications 26 (2017), no.7, 1750035, 22 pp. 

\bibitem{Maple} Clark, W.E., {\rm Maple Worksheets,} \url{http://shell.cas.usf.edu/~saito/SphericalQuandle/Maple-Files-TopQ/}

  \bibitem{DJ} Duistermatt, J.J.; Kolk, J.A.C.,  {\rm Lie Groups}, Springer-Verlag, 2000.

\bibitem{Eis-cover}
Eisermann, M., {\it Quandle coverings and their galois correspondence,} 
Fund. Math. 225 (2007) 103--167.

\bibitem{Eis-colpoly} 
Eisermann, M., {\it Knot colouring polynomials,}
 Pacific J. Math., 231 (2007) 305--336. 


\bibitem{Eis-unknot}
Eisermann, M., {\it Homological characterization of the unknot,} 
J. Pure Appl. Algebra, {177} (2003)  
131--157. 


\bibitem{HK}
Heusener, M.; Klassen, E., 
{\it Deformations of dihedral representations,}
Proc. Amer. Math. Soc. 123 (1997) 3039--3047. 



\bibitem{Ino} Inoue, A., 
{\it  On colorability of knots by rotations, torus knot and PL trochoid,}
Topology Appl. 183 (2015) 36--44. 


\bibitem{Joyce} Joyce, D.,
{\it A classifying invariant of knots, the knot quandle,}
J. Pure Appl. Alg., 23 (1983) 37--65.


\bibitem{Kawauchi} Kawauchi, A.,
{\it A survey of knot theory}, 
Birkhauser-Verlag, 1996.

\bibitem{Klassen} Klassen, E.P., {\it Representations of Knot Groups in SU(2)}, Trans. Amer. Math. Soc.,  326(2) (1991) 795--828.


\bibitem{Kuipers}
Kuipers, J.B.,
{\rm Quaternions and Rotation Sequences,}
Princeton University Press, 1999.


\bibitem{Mat} Matveev, S., 
{\it Distributive groupoids in knot theory.} (Russian) Mat. Sb. (N.S.)
119(161) (1982) 
78--88 (160). 

\bibitem{Nieb}
Niebrzydowski, M., 
{\it On colored quandle longitudes and its applications to tangle embeddings and virtual knots,}
J. Knot Theory Ramifications, 15 (2006) 1049--1059. 


\bibitem{Porteous} Porteous, I.R., {\rm Topological Geometry}, Cambridge University Press, 2nd edition, 1969.


\bibitem{Rub} Rubinsztein, R.,
{\it Topological invariants and invariants of links,}
 J. Knot Theory Ramifications,
16 (2007) 789--808.


\end{thebibliography}
\end{document}